\documentclass[thmsa,onecolumn,11pt]{article}
\usepackage{amssymb}
\usepackage{amsfonts}
\usepackage{graphicx}
\usepackage{amsmath}

\newtheorem{theorem}{Theorem}

\newtheorem{corollary}[theorem]{Corollary}

\newtheorem{lemma}[theorem]{Lemma}

\newtheorem{proposition}[theorem]{Proposition}

\newenvironment{proof}[1][Proof]{\textbf{#1.} }{\ \rule{0.5em}{0.5em}}

\begin{document}

\title{\bf Hoeffding spaces and Specht modules}
\author{Giovanni PECCATI\thanks{
\'{E}quipe \textit{Modal'X}. Universit\'{e} Paris Ouest - Nanterre La D\'{e}%
fense. 200, Avenue de la R\'{e}publique, 92000 Nanterre France. E-mail:
\texttt{giovanni.peccati@gmail.com}} \ \ and  \ Jean-Renaud PYCKE\thanks{%
Department of mathematics, University of \'{E}vry, France. E-mail: \texttt{%
jrpycke@maths.univ-evry.fr}}}
\date{April 20, 2009}
\maketitle

\begin{abstract}
It is proved that each Hoeffding space associated with a random permutation
(or, equivalently, with extractions without replacement from a finite
population) carries an irreducible representation of the symmetric group,
equivalent to a two-block Specht module.

\smallskip


\smallskip

\end{abstract}

\section{Introduction}

Let $X_{\left( m\right) }=\left( X_{1},...,X_{m}\right) $ ($m\geq 2$) be a
sample of random observations. According e.g. to \cite{PecAop}, we say that $%
X_{\left( m\right) }$ is \textsl{Hoeffding-decomposable}\ if every symmetric
statistic of $X_{\left( m\right) }$ can be written as an orthogonal sum of
symmetric $U$-statistics with degenerated kernels of increasing orders. In
the case where $X_{\left( m\right) }$ is composed of i.i.d. random
variables, Hoeffding decompositions are a classic and very powerful tool for
obtaining limit theorems, as $m\rightarrow \infty $, for sequences of
general symmetric statistics of the vectors $X_{\left( m\right) }$. See e.g.
\cite{Serf}, or the references indicated in the introduction to \cite{PecAop}%
, for further discussions in this direction.

In recent years, several efforts have been made in order to provide a
characterization of Hoeffding decompositions associated with \textsl{%
exchangeable} (and not necessarily independent) vectors of observations. See
El-Dakkak and Peccati \cite{ElPec} and Peccati \cite{PecAop} for some
general statements; see Bloznelis \cite{Bloz2005}, Bloznelis and G\"{o}tze
\cite{BG1, BG2} and Zhao and Chen \cite{ZhaoChen} for a comprehensive
analysis of Hoeffding decompositions associated with extractions without
replacement from a finite population.

In the present note, we are interested in building a new explicit connection
between the results of \cite{BG1, BG2, ZhaoChen} and the irreducible
representations of the symmetric groups $\mathfrak{S}_{n}$, $n\geq 2$. In
particular, our main result is the following.

\begin{theorem}
\label{T : Main Intro}Let $1\leq m\leq n/2$, and let $X_{\left( m\right)
}=\left( X\left( 1\right) ,...,X\left( m\right) \right) $ be a random vector
obtained as the first $m$ extractions without replacement from a population
of $n$ individuals. For $l=1,...,m$, let $SH_{l}$ be the $l$th symmetric
Hoeffding space associated with $X_{\left( m\right) }$ (that is, $SH_{l}$ is
the vector space of all symmetric $U$-statistics with a completely
degenerated kernel of order $l$). Then, for every $l=1,...,m$, there exists
an action of $\mathfrak{S}_{n}$ on $SH_{l}$, such that $SH_{l}$ is an
irreducible representation of $\mathfrak{S}_n$. This representation is equivalent to
a Specht module of shape $\left( n-l,l\right) .$
\end{theorem}

We refer the reader to the forthcoming Section \ref{S : Backg} for some
basic results on the representations of the symmetric group and two-block
Specht modules. We will see that Theorem \ref{T : Main Intro} provides
\textit{de facto }a new probabilistic characterization of two-block Specht
modules, as well as some original insights into the combinatorial structure
of Hoeffding spaces. Observe that the case
where $n/2<m\leq n$ can be reduced to the framework of present paper by standard
arguments (see for instance \cite[Proposition 1]{BG1}). One should note that a connection between
decompositions of symmetric statistics and representations of $\mathfrak{S}%
_{n}$ is already sketched in Diaconis' celebrated monograph \cite{Diaconis}:
in particular, the results of the present paper can be regarded as a
probabilistic counterpart to the \textit{spectral analysis on homogeneous
spaces }developed in Chapters 7 and 8 of \cite{Diaconis}.

The rest of this note is organized as follows. In Section \ref{S : Backg} we
provide some background on the representations of the symmetric group.
Sections \ref{S : URP} and \ref{S : Hoeff} focus, respectively, on uniform
random permutations and Hoeffding spaces. Section \ref{S : HS//SM} contains
the statements and proofs of our main results.

\section{Background\label{S : Backg}}

For future reference, we recall that a $k$-\textsl{block partition} of the
integer $n\geq 2$ is a $k$-dimensional vector of the type $\lambda =\left(
\lambda _{1},...,\lambda _{k}\right) $, such that: (i) each $\lambda _{i}$
is a strictly positive integer, (ii) $\lambda _{i}\geq \lambda _{i+1}$, and
(iii) $\lambda _{1}+\cdot \cdot \cdot +\lambda _{k}=n$. One sometimes writes
$\lambda \vdash n$ to indicate that $\lambda $ is a partition of $n$.

We also write $\left[ n\right] =\left\{ 1,...,n\right\} $ to indicate the
set of the first $n$ positive integers. Finally, given a finite set $A$, we
denote by $\mathfrak{S}_{A}$ the group of all permutations of $A$, and we
use the shorthand notation $\mathfrak{S}_{\left[ n\right] }=\mathfrak{S}_{n}$%
, $n\geq 1$. In other words, when writing $x\in \mathfrak{S}_{A}$,
we mean that $$x : A\rightarrow A : a\mapsto x(a)$$ is a bijection from $A$ to itself.

\subsection{Some structures associated with two-block partitions}

We now introduce some classic definitions and notation related to tableaux
and tabloids; see Sagan \cite[Chapter 2]{Sagan} (from which we borrow most
of our terminology and notational conventions) for any unexplained concept
or result. For the rest of the section, we fix two integers $n$ and $m$,
such that $1\leq m\leq n/2$. Observe that $n-m\geq m$, and therefore the
vector $\left( n-m,m\right) $ is a two-block partition\textsl{\ }of the
integer $n$.

\bigskip

\noindent \textbf{Remark. }It is sometimes useful to adopt a graphical
representation of tableaux and tabloids by means of \textsl{Ferrer diagrams}%
. Since we uniquely deal with two-block tableaux and tabloids, and
for the sake of brevity, in what follows we shall not make use of this representation. See
e.g. \cite[Section 2.1]{Sagan} for a complete discussion of this point.

\bigskip

The following objects will be needed in the sequel.

\begin{description}
\item[--] A (Young) \textsl{tableau} $t$ of shape $\left( n-m,m\right) $ is
a pair $t=\left( i_{\left( n-m\right) };j_{\left( m\right) }\right) $ of
\textsl{ordered} vectors of the type $i_{\left( n-m\right) }=\left(
i_{1},...,i_{n-m}\right) $, $j_{\left( m\right) }=\left(
j_{n-m+1},...,j_{n}\right) $ such that $\left\{
i_{1},...,i_{n-m},j_{n-m+1},...,j_{n}\right\} =\left[ n\right] $, that is,
the union of the entries of $i_{\left( n-m\right) }$ and $j_{\left( m\right)
}$ coincides with the first $n$ integers (with no repetitions).

\item[--] The set of the \textsl{columns} of the tableau $t=\left( i_{\left(
n-m\right) };j_{\left( m\right) }\right) $, noted $\left\{
C_{1},...,C_{n-m}\right\} $, is the collection of (i) the ordered pairs
\begin{equation}
C_{1}=\left( i_{1},j_{n-m+1}\right) ,...,C_{m}=\left( i_{m},j_{n}\right)
\label{ColonneA}
\end{equation}%
(that is, the pairs composed of the first $m$ entries of $i_{\left(
n-m\right) }$ and the entries of $j_{\left( m\right) }$), and (ii) the
remaining singletons of $i_{\left( n-m\right) }$, that is,%
\begin{equation}
C_{m+1}=\text{ }i_{m+1},...,C_{n-m}=i_{n-m}\text{.}  \label{ColonneB}
\end{equation}

\item[--] For $l=1,...,n$, we write $V^{\left( n-l,l\right) }$ to indicate
the class of the $\binom{n}{l}$ subsets of $\left[ n\right] $ of size equal
to $l$. This slightly unusual notation has been chosen in order to stress
the connection between the set $V^{\left( n-l,l\right) }$ and the $\mathfrak{%
S}_{n}$-modules $M^{\left( n-l,l\right) }$ ($l\leq m$) to be defined below.
The elements of $V^{\left( n-l,l\right) }$ are denoted by ${\bf a}%
_{\left( l\right) }$, ${\bf b}%
_{\left( l\right) }$, $\mathbf{i}%
_{\left( l\right) }$, $\mathbf{j}_{\left( l\right) }$,..., and so on.

\item[--] A \textsl{tabloid }of shape $\left( n-m,m\right) $ is a two-block
partition of the set $\left[ n\right] $, of the type
\begin{equation}
\gamma =\{\mathbf{a }_{\left( n-m\right) };\mathbf{b }_{\left(
m\right) }\}=\{\left\{ a _{1},...,a _{n-m}\right\} ;\left\{ b
_{n-m+1},...,b _{n}\right\} \}.  \label{gobbi}
\end{equation}%
Of course, a tabloid $\gamma $ of shape $\left( n-m,m\right) $ as in (\ref%
{gobbi}) is completely determined by the specification of set $\mathbf{b
}_{\left( m\right) }=\left\{ b_{n-m+1},...,b_{n}\right\} \in
V^{\left( n-m,m\right) }$; to emphasize this dependence, we shall sometimes
write $\gamma =\gamma (\mathbf{b }_{\left( m\right) }).$ Note that the
mapping $\mathbf{b }_{\left( m\right) }\mapsto \gamma (\mathbf{b }%
_{\left( m\right) })$ is a bijection between $V^{\left( n-m,m\right) }$ and
the class of all tabloids of shape $\left( n-m,m\right) $.

\item[--] Given a tableau $t=\left( i_{\left( n-m\right) };j_{\left(
m\right) }\right) $ of shape $\left( n-m,m\right) $, we write $\left\{
t\right\} =\{\mathbf{i}_{\left( n-m\right) };\mathbf{j}_{\left( m\right) }\}$
(observe the boldface!) to indicate the tabloid defined by $\mathbf{i}%
_{\left( n-m\right) }=\left\{ i_{1},...,i_{n-m}\right\} $ and $\mathbf{j}%
_{\left( m\right) }=\left\{ j_{n-m+1},...,j_{n}\right\} $. In other words, $%
\{t\}$ is obtained as the two-block partition composed of the collection of
the entries of $i_{\left( n-m\right) }$ and the collection of the entries of
$j_{\left( m\right) }$. With the notation introduced at the previous point,
one has that\thinspace $\left\{ t\right\} =\gamma (\mathbf{j}_{\left(
m\right) })$.
\end{description}

\bigskip

\noindent \textbf{Example. }Let $n=5$ and $m=2$. Then, a tableau of shape $%
\left( 3,2\right) $ is $t=\left( i_{\left( 3\right) };j_{\left( 2\right)
}\right) $, where $i_{\left( 3\right) }=\left( 2,1,3\right) $ and $j_{\left(
2\right) }=\left( 5,4\right) $. The columns of $t$ are $C_{1}=\left(
2,5\right) $, $C_{2}=\left( 1,4\right) $ and $C_{3}=3$. The associated
tabloid is $\left\{ t\right\} =\{\mathbf{i}_{\left( 3\right) };\mathbf{j}%
_{\left( 2\right) }\}$, where $\mathbf{i}_{\left( 3\right) }=\left\{
1,2,3\right\} \in V^{\left( 2,3\right) }$ and $\mathbf{j}_{\left( 2\right)
}=\left\{ 4,5\right\} \in V^{\left( 3,2\right) }$.

\subsection{Actions of $\mathfrak{S}_{n}$}

Fix as before $n\geq 2$ and $1\leq m\leq n/2$.

\bigskip

\noindent \underline{\textit{Actions on tableaux}}. For every $x\in
\mathfrak{S}_{n}$ and every tableaux $t=\left( i_{\left( n-m\right)
};j_{\left( m\right) }\right) $, the action of $x$ on $t$ is defined as
follows:%
\begin{equation}
xt=\left( xi_{\left( n-m\right) };xj_{\left( m\right) }\right) \text{,}
\label{zzz}
\end{equation}%
where $xi_{\left( n-m\right) }=\left( x\left( i_{1}\right) ,...,x\left(
i_{n-m}\right) \right) $ and $xj_{\left( m\right) }=\left( x\left(
j_{n-m+1}\right) ,...,x\left( j_{n}\right) \right) $.

\bigskip

\noindent \underline{\textit{Actions on tabloids}}. For every $x\in
\mathfrak{S}_{n}$ and every tabloid $\gamma (\mathbf{b }_{\left(
m\right) })=\{\mathbf{a}_{\left( n-m\right) };$ $\mathbf{b}%
_{\left( m\right) }\}$, we set
\begin{eqnarray}
x\gamma (\mathbf{b}_{\left( m\right) }) &=&x\{\left\{ a
_{1},...,a_{n-m}\right\} ;\left\{ b_{n-m+1},...,b_{n}\right\} \}  \label{google} \\
&=&\{\left\{ x\left( a_{1}\right) ,...,x\left( a_{n-m}\right)
\right\} ;\left\{ x\left( b_{n-m+1}\right) ,...,x\left( b_{n}\right) \right\} \}  \notag
\end{eqnarray}%
In particular, for every tableau $t$, one has $x\left\{ t\right\} =\left\{
xt\right\} $.

\bigskip

\noindent \underline{$\mathfrak{S}_{n}$-\textit{modules}}. The symmetric
group $\mathfrak{S}_{n}$ acts on $V^{\left( n-m,m\right) }$ in the standard
way, namely: for every $x\in \mathfrak{S}_{n}$ and for every $\mathbf{j}%
_{\left( m\right) }=\left\{ j_{1},...,j_{m}\right\} \in V^{\left(
n-m,m\right) }$,%
\begin{equation}
x\mathbf{j}_{\left( m\right) }=\left\{ x\left( j_{1}\right) ,...,x\left(
j_{m}\right) \right\} .  \label{actionSym0}
\end{equation}

\bigskip

\noindent \textbf{Remark. }By combining the above introduced notational
conventions, one sees that, for every $x\in \mathfrak{S}_{n}$ and for every $%
\mathbf{j}_{\left( m\right) }=V^{\left( n-m,m\right) }$,
\begin{equation*}
x\gamma (\mathbf{j}_{\left( m\right) })=\gamma (x\mathbf{j}_{\left( m\right)
})\text{,}
\end{equation*}%
that is, $x$ transforms the tabloid generated by $\mathbf{j}_{\left(
m\right) }$ into the tabloid generated by $x\mathbf{j}_{\left( m\right) }$.
Also, if $t=\left( i_{\left( n-m\right) };j_{\left( m\right) }\right) $,
then, for every $x\in \mathfrak{S}_{n}$,
\begin{equation*}
x\left\{ t\right\} =\left\{ xt\right\} =x\gamma (\mathbf{j}_{\left( m\right)
})=\gamma (x\mathbf{j}_{\left( m\right) }).
\end{equation*}

\bigskip

The complex vector space of all complex-valued functions on $V^{\left( n-m,m\right)
} $ is written $L\left( V^{\left( n-m,m\right) }\right) $. Plainly, the
space $L\left( V^{\left( n-m,m\right) }\right) $ has dimension $\binom{n}{m}$%
, and a basis of $L\left( V^{\left( n-m,m\right) }\right) $ is given by the
collection $\{\mathbf{1}_{\mathbf{j}_{\left( m\right) }}:\mathbf{j}_{\left(
m\right) }\in V^{\left( n-m,m\right) }\}$, where $\mathbf{1}_{\mathbf{j}%
_{\left( m\right) }}\left( \mathbf{k}_{\left( m\right) }\right) =1$ if $%
\mathbf{k}_{\left( m\right) }=\mathbf{j}_{\left( m\right) }$ and $\mathbf{1}%
_{\mathbf{j}_{\left( m\right) }}\left( \mathbf{k}_{\left( m\right) }\right)
=0$ otherwise. The group $\mathfrak{S}_{n}$ acts on $L\left( V^{\left(
n-m,m\right) }\right) $ as follows: for $x\in \mathfrak{S}_{n}$, $\mathbf{k}%
_{\left( m\right) }\in V^{\left( n-m,m\right) }$ and $f\in L\left( V^{\left(
n-m,m\right) }\right) $,
\begin{eqnarray}
xf\left( \mathbf{k}_{\left( m\right) }\right) &=&f\left( x^{-1}\mathbf{k}%
_{\left( m\right) }\right) \text{, so that, in particular,}
\label{actionSym} \\
x\mathbf{1}_{\mathbf{j}_{\left( m\right) }} &=&\mathbf{1}_{x\mathbf{j}%
_{\left( m\right) }}\text{, \ \ }\mathbf{j}_{\left( m\right) }\in V^{\left(
n-m,m\right) }.  \notag
\end{eqnarray}%
When endowed with the action (\ref{actionSym}), the set $L\left( V^{\left(
n-m,m\right) }\right) $ carries a representation of $\mathfrak{S}_{n}$. In
this case, we say that $L\left( V^{\left( n-m,m\right) }\right) $ is the
\textsl{permutation module}\textbf{\ }associated with $\left( n-m,m\right) $%
, and we use the customary notation $L\left( V^{\left( n-m,m\right) }\right)
=M^{\left( n-m,m\right) }$ (see \cite[Section 2.1]{Sagan}).

\bigskip

\noindent \textbf{Remark. }Our definition of the permutation modules $%
M^{\left( n-m,m\right) }$ slightly differs from the one given e.g. in \cite[%
Definition 2.1.5]{Sagan}. Indeed, we define $M^{\left( n-m,m\right) }$ as
the vector space spanned by all indicators of the type $\mathbf{1}_{\mathbf{j%
}_{\left( m\right) }}$, $\mathbf{j}_{\left( m\right) }\in V^{\left(
n-m,m\right) }$, endowed with the action (\ref{actionSym}), whereas in the
above quoted reference $M^{\left( n-m,m\right) }$ is the space of all formal
linear combinations of tabloids of shape $\left( n-m,m\right) $, endowed
with the canonical extension of the action (\ref{google}). The two
definitions are equivalent, in the sense that they give rise to two
isomorphic $\mathfrak{S}_{n}$-modules. We will see that the definition of $%
M^{\left( n-m,m\right) }$ chosen in this paper allows a more transparent
connection with the theory of $U$-statistics based on random permutations.

\subsection{A decomposition of $M^{\left( n-m,m\right) }$}

We recall that the dual of $\mathfrak{S}_{n}$ coincides with the set $%
\left\{ \left[ S^{\lambda }\right] :\lambda \vdash n\right\} $, where $\left[
S^{\lambda }\right] $ is the equivalence class of all irreducible
representations of $\mathfrak{S}_{n}$ that are equivalent to a Specht module
of index $\lambda $ (see again \cite[Section 2.1]{Sagan}). For every $%
\lambda \vdash n$, we will denote by $\chi ^{\lambda }$ the character
associated with the class $\left[ S^{\lambda }\right] $, whereas $D_{\lambda }$
is the associate dimension. Observe that $\chi ^{\lambda }\in \mathbb{Z}$
for every $\lambda $ (see e.g. \cite[Section 13.1]{Serre}), and $D_{\lambda
} $ equals the number of standard tableaux (that is, tableaux with
increasing rows and columns) of shape $\lambda $. In particular $D_{\left(
n-1,1\right) }=n-1$ (see \cite[Section 2.5]{Sagan}).

The next result ensures that the module $M^{\left( n-m,m\right) }$ is
reducible. This fact is well-known (see e.g. \cite[Example 14.4, p. 52]%
{James} or \cite[pp. 134-139]{Diaconis}), and a proof is added here for the
sake of completeness.

\begin{proposition}\label{P : dec}
There exists a unique decomposition of $M^{\left( n-m,m\right) }$ of the type%
\begin{equation}
M^{\left( n-m,m\right) }=K_{0}^{\left( n-m,m\right) }\oplus K_{1}^{\left(
n-m,m\right) }\oplus \cdot \cdot \cdot \oplus K_{m}^{\left( n-m,m\right) }%
\text{.}  \label{spechtDec}
\end{equation}%
Where the vector spaces (endowed with the action of $\mathfrak{S}_{n}$
described in (\ref{actionSym})) $K_{l}^{\left( n-m,m\right) }$ are such that
$K_{0}^{\left( n-m,m\right) }\in \left[ S^{\left( n\right) }\right] $, and $%
K_{l}^{\left( n-m,m\right) }\in \left[ S^{\left( n-l,l\right) }\right] $, $%
l=1,...,m$.
\end{proposition}

\begin{proof}
It is sufficient to prove that
\begin{equation*}
M^{\left( n-m,m\right) }\cong S^{\left( n\right) }\oplus
\bigoplus\limits_{l=1}^{m}S^{\left( n-l,l\right) }\text{,}
\end{equation*}%
where \textquotedblleft\ $\cong $ \textquotedblright\ indicates equivalence
between representations of $\mathfrak{S}_{n}$. According Young's Rule (see
e.g. \cite[Th. 2.11.2]{Sagan}), we know that%
\begin{equation*}
M^{\left( n-m,m\right) }\cong S^{\left( n\right) }\oplus
\bigoplus\limits_{l=1}^{m}K_{n,l,m}S^{\left( n-l,l\right) }\text{,}
\end{equation*}%
where the integers $K_{n,l,m}$ (known as \textsl{Kostka numbers}) count the
number of generalized semistandard tableaux of shape $\left( n-l,l\right) $
and type $\left( n-m,m\right) $. This is equivalent to saying $K_{n,l,m}$
counts the ways of arranging $n-m$ copies of $1$ and $m$ copies of $2$ in a
Ferrer diagram of shape $\left( n-l,l\right) $, in such a way that the rows
of the diagram are weakly increasing and the columns are strictly
increasing. Since there is just one way of doing this, one infers that $%
K_{n,l,m}=1$, and the proof is concluded.
\end{proof}

\bigskip

\noindent \textbf{Remarks. }(i)\textbf{\ }(\textit{Definition of two-block
Specht modules}) For the sake of completeness, we recall here the definition
of the modules $S^{\left( n\right) }$ and $S^{\left( n-m,m\right) }$, $1\leq
m\leq n/2$. First of all, one has that $S^{\left( n\right) }=\mathbb{C}$,
and therefore $\left[ S^{\left( n\right) }\right] $ is the class of
representations of $\mathfrak{S}_{n}$ that are equivalent to the trivial
representation. Now fix $1\leq m\leq n/2$. For every tableau $t=(
i_{\left( n-m\right) };j_{\left( m\right) }) $, define the columns $%
C_{1},...,C_{n-m}$ according to (\ref{ColonneA}) and (\ref{ColonneB}). Then,
(a) for every $l=1,...,m$, write $\kappa _{C_{l}}$ for the formal operator
\begin{equation*}
\kappa _{C_{l}}=\text{Id.}-(i_{l}\rightarrow j_{l})\text{,}
\end{equation*}%
where $\left( i_{l}\rightarrow j_{l}\right) $ indicates the element of $%
\mathfrak{S}_{n}$ given by the translation sending $i_{l}$ to $j_{l}$, and
(b) define the composed operator $\kappa _{t}=\kappa _{C_{1}}\kappa
_{C_{2}}\cdot \cdot \cdot \kappa _{C_{m}}$. Then, the Specht module of shape
$\left( n-m,m\right) $ is the $\mathfrak{S}_{n}$-invariant subspace of $%
M^{\left( n-m,m\right) }$ spanned by the elements of the type%
\begin{equation}
\kappa _{t}\mathbf{1}_{\mathbf{j}_{\left( m\right) }}\text{, \ \ where }%
t=\left( i_{\left( n-m\right) };j_{\left( m\right) }\right) \text{ is a
tableau;}  \label{spechtgen}
\end{equation}%
note that, in the formula (\ref{spechtgen}), $t$ and $\mathbf{j}_{\left(
m\right) }$ are related by the fact that $t=\left( i_{\left( n-m\right)
};j_{\left( m\right) }\right) $, and $\left\{ t\right\} =\{\mathbf{i}%
_{\left( n-m\right) };\mathbf{j}_{\left( m\right) }\}$.

\noindent (ii) Consider for instance the case $n=6$ and $m=2$, and select
the tableau $t=\left\{ \left( 1,2,3,4\right) ;\left( 5,6\right) \right\} $.
One has that $\mathbf{j}_{\left( 2\right) }=\left\{ 5,6\right\} $,
\begin{equation*}
\kappa _{t}=\left( \text{Id.}-\left( 1\rightarrow 5\right) \right) \left(
\text{Id.}-\left( 2\rightarrow 6\right) \right) \text{,}
\end{equation*}%
and one deduces that an element of $S^{\left( 4,2\right) }$ is given by%
\begin{equation*}
\kappa _{t}\mathbf{1}_{\mathbf{j}_{\left( 2\right) }}=\mathbf{1}_{\left\{
5,6\right\} }-\mathbf{1}_{\left\{ 1,6\right\} }-\mathbf{1}_{\left\{
5,2\right\} }+\mathbf{1}_{\left\{ 1,2\right\} }\text{.}
\end{equation*}

\noindent (iii) By recurrence, one deduces from Proposition \ref{P : dec} that the dimension of $K_l^{(n-m,m)}$, and therefore of $S^{(n-l,l)}$, is $D_{(n-l,l)}=\binom{n}{l} - \binom{n}{l-1}$, $l\leq n/2$.

\noindent (iv) From the previous discussion, we infer that $%
K_{0}^{\left( n-m,m\right) }=S^{\left( n\right) }=\mathbb{C}$.

\section{Uniform random permutations\label{S : URP}}

Fix $n\geq 2$. We consider a uniform \textsl{random permutation}\textbf{\ }$%
X $ of $\left[ n\right] $. This means that $X=X\left( \omega \right) $ is a
random element with values in $\mathfrak{S}_{n}$, defined on some finite
probability space $\left( \Omega ,\mathcal{F},\mathbf{P}\right) $ and such
that, $\forall x\in \mathfrak{S}_{n}$, $\mathbf{P}\left( X=x\right) =\left(
n!\right) ^{-1}$. For $1\leq m\leq n/2$ as before, we will write $%
X_{\left( m\right) }\left( \omega \right) =\left( X\left( 1\right)
,...,X\left( m\right) \right) \left( \omega \right) $, and also, for every $y\in\mathfrak{S}_n$, $(Xy)_{(m)} = \{Xy(1),...,Xy(m)\}$. 
Observe that $Xy$ indicates the product of the deterministic permutation $y$ with the random permutation $X$. It is clear that $%
X_{\left( m\right) }$ is an \textsl{exchangeable} vector, having the law of
the first $m$ extractions without replacement from the set $\left[ n\right] $
(see e.g. Aldous \cite{Aldous} for any unexplained notion about
exchangeability). A random variable $T$ is called a (complex-valued) \textsl{%
symmetric statistic} of $X_{\left( m\right) }$ if $T$ has the form
\begin{equation*}
T=f\left( \left\{ X\left( 1\right) ,...,X\left( m\right) \right\} \right)
\text{, \ \ for some }f\in L\left( V^{\left( n-m,m\right) }\right) \text{.}
\end{equation*}%
In other words, a symmetric statistic is a random variable deterministically
depending on the realization of $X_{\left( m\right) }$ as a non-ordered set.
Note that, by a slight abuse of notation, in what follows we will write $%
f\left( \left\{ X\left( 1\right) ,...,X\left( m\right) \right\} \right)
=f\left( X_{\left( m\right) }\right) $ (other analogous conventions will be
tacitly adopted).

We also write $L_{s}^{2}\left( X_{\left( m\right) }\right) $ to indicate the
Hilbert space of symmetric statistics of $X_{\left( m\right) }$, endowed
with the inner product
\begin{eqnarray}
\left\langle f_{1}\left( X_{\left( m\right) }\right) ,f_{2}\left( X_{\left(
m\right) }\right) \right\rangle _{\mathbf{P}} &=&\mathbf{E}\left[
f_{1}\left( X_{\left( m\right) }\right) \overline{f_{2}\left( X_{\left(
m\right) }\right) }\right]  \label{innPr1} \\
&=&\frac{1}{n!}\sum_{x\in \mathfrak{S}_{n}}f_{1}\left( x\left\{
1,...,m\right\} \right) \overline{f_{2}\left( x\left\{ 1,...,m\right\}
\right) }  \label{pomi} \\
&=&\dbinom{n}{m}^{-1}\sum_{\mathbf{k}_{\left( m\right) }\in V^{\left(
n-m,m\right) }}f_{1}\left( \mathbf{k}_{\left( m\right) }\right) \overline{%
f_{2}\left( \mathbf{k}_{\left( m\right) }\right) }\text{.}  \notag
\end{eqnarray}%
Since the sum in (\ref{pomi}) runs over the whole set $\mathfrak{S}_{n}$, it
is clear that $\left\langle \cdot ,\cdot \right\rangle _{\mathbf{P}}$
induces a $\mathfrak{S}_{n}$-invariant inner product on $M^{\left( n-m,m\right) }$ given by
\begin{equation}
\left\langle f_{1},f_{2}\right\rangle _{\left( n-m,m\right) }=\left\langle
f_{1}\left( X_{\left( m\right) }\right) ,f_{2}\left( X_{\left( m\right)
}\right) \right\rangle _{\mathbf{P}},\text{ \ \ }f_{1},f_{2}\in M^{\left(
n-m,m\right) };  \label{innPr2}
\end{equation}%
in particular, the $\mathfrak{S}_n$-invariance of $\left\langle \cdot ,\cdot
\right\rangle _{\left( n-m,m\right) }$ yields that the spaces $K_{i}^{\left(
n-m,m\right) }$ and $K_{j}^{\left( n-m,m\right) }$ are orthogonal with
respect to $\left\langle \cdot ,\cdot \right\rangle _{\left( n-m,m\right) }$
for every $0\leq i\neq j\leq m$.

With every $f\in M^{\left( n-m,m\right) }$, we associate the $\mathfrak{S}%
_{n}$-indexed stochastic process%
\begin{equation*}
Z_{f}\left( x,\omega \right) =Z_{f}\left( x\right) :=f\left( xX_{\left(
m\right) }\right) \text{, \ \ }x\in \mathfrak{S}_{n}\text{,}
\end{equation*}%
and, for every $\lambda \vdash n$, we define
\begin{eqnarray}
Z_{f}^{\lambda }\left( x,\omega \right) &=&Z_{f}^{\lambda }\left( x\right) :=%
\frac{D_{\lambda }}{n!}\sum_{g\in \mathfrak{S}_{n}}\chi ^{\lambda }\left(
g\right) f\left( (g^{-1}x)X_{\left( m\right) }\right)  \label{defChProj} \\
f^{\lambda }\left( \mathbf{l}_{\left( m\right) }\right) &=&\frac{D_{\lambda }%
}{n!}\sum_{x\in \mathfrak{S}_{n}}\chi ^{\lambda }\left( x\right) f\left(
x^{-1}\mathbf{l}_{\left( m\right) }\right) ,\text{ \ \ }\mathbf{l}_{\left(
m\right) }\in V^{\left( n-m,m\right) }\text{,}  \notag
\end{eqnarray}%
so that $f^{\lambda }\left( X_{\left( m\right) }\right) =Z_{f}^{\lambda
}\left( e\right) $, where $e$ is the identity element in $\mathfrak{S}_{n}$.

The following facts will be used in the subsequent analysis. The proofs are
standard and omitted -- see e.g. the results from \cite{PecPycke} and \cite{Serre} evoked below for further details.

\begin{description}
\item[(a)] Since (\ref{spechtDec}) holds, $f^{\lambda }=0$ for every $f\in
M^{\left( n-m,m\right) }$ if and only if $\lambda $ is different from $%
\left( n-l,l\right) $, $l=0,...,m$ (see e.g. \cite[Theorem 8, Section 2.6]%
{Serre}) and moreover: $f^{(n)}\in K_{0}^{\left( n-m,m\right) }$ and, for every $l=1,...,m$, $f^{\left( n-l,l\right) }\in
K_{l}^{\left( n-m,m\right) }$ (as defined in (\ref{spechtDec})).

\item[(b)] Thanks to exchangeability, for every $f\in M^{\left( n-m,m\right)
}$ the class
\begin{equation*}
\left\{ Z_{f},Z_{f}^{\left( n-l,l\right) }:l=0,...,m\right\} \text{,}
\end{equation*}%
has a $\mathfrak{S}_{n}$-invariant law, with respect to the canonical action
of $\mathfrak{S}_{n}$ on itself (i.e., $x\cdot y=xy$, $x,y\in \mathfrak{S}%
_{n}$).

\item[(c)] Due to the orthogonality of isotypical spaces (see e.g. (see \cite[Theorem
4.4.5]{DuiKolk}, and also \cite[Theorem 4-3]{PecPycke}), for every $x,y\in \mathfrak{S}_{n}$, $%
f,h\in M^{\left( n-m,m\right) }$ and $0\leq i\neq j\leq m$,%
\begin{eqnarray}
&& \mathbf{E}\left[ Z_{f}^{\left( n-i,i\right) }\left( x\right) \overline{%
Z_{h}^{\left( n-j,j\right) }\left( y\right) }\right] =\mathbf{E}\left[
f^{\left( n-i,i\right) }\left( xX_{\left( m\right) }\right) \overline{%
h^{\left( n-j,j\right) }\left( yX_{\left( m\right) }\right) }\right] \label{Ostiano}\\
&& \mathbf{E}\left[
f^{\left( n-i,i\right) }\left( (Xx)_{\left( m\right) }\right) \overline{%
h^{\left( n-j,j\right) }\left( (Xy)_{\left( m\right) }\right) }\right] = 0,
\label{OS1}
\end{eqnarray}
where, here and in the sequel (by a slight abuse of notation) we use the convention $(n-0,0)=(n)$.
\item[(d)] Due to \cite[Theorem 4-4]{PecPycke} and point (a) above, for
every $x\in \mathfrak{S}_{n}$ and every $f\in M^{\left( n-m,m\right) }$,
\begin{equation}
Z_{f}\left( x\right) =Z_{f}^{\left( n\right) }\left( x\right)
+\sum_{l=1}^{m}Z_{f}^{\left( n-l,l\right) }\left( x\right) ,  \label{OS2}
\end{equation}%
where $Z_{f}^{\left( n\right) }\left( x\right) =\mathbf{E}\left[ Z_{f}\left(
x\right) \right] =\mathbf{E}\left[ f\left( X_{\left( m\right) }\right) %
\right] .$ In particular,%
\begin{equation}
f\left( X_{\left( m\right) }\right) =\mathbf{E}\left[ f\left( X_{\left(
m\right) }\right) \right] +\sum_{l=1}^{m}f^{\left( n-l,l\right) }\left(
X_{\left( m\right) }\right)  \label{OStat}
\end{equation}%
and therefore, for every $f,h\in M^{\left( n-m,m\right) }$,
\begin{equation}
\mathbf{E}\left[ f\left( X_{\left( m\right) }\right) \overline{h\left(
X_{\left( m\right) }\right) }\right] =\mathbf{E}\left[ f\left( X_{\left(
m\right) }\right) \right] \overline{\mathbf{E}\left[ h\left( X_{\left(
m\right) }\right) \right] }+\sum_{l=1}^{m}\mathbf{E}\left[ f^{\left(
n-l,l\right) }\left( X_{\left( m\right) }\right) \overline{h^{\left(
n-l,l\right) }\left( X_{\left( m\right) }\right) }\right]  \label{OScovr}
\end{equation}

\item[(e)] Due to \cite[Theorem 5-1]{PecPycke}, for every $0\leq i\neq j\leq
m$ and $f,h\in M^{\left( n-m,m\right) }$,%
\begin{equation}
\sum_{x\in \mathfrak{S}_{n}}Z_{f}^{\left( n-i,i\right) }\left( x,\omega
\right) \overline{Z_{h}^{\left( n-j,j\right) }\left( x,\omega \right) }%
=\sum_{x\in \mathfrak{S}_{n}}f^{\left( n-i,i\right) }\left( xX_{\left(
m\right) }\right) \overline{h^{\left( n-j,j\right) }\left( xX_{\left(
m\right) }\right) }=0.  \label{OS3}
\end{equation}
\end{description}

\section{Hoeffding spaces\label{S : Hoeff}}

We now define a class of subspaces of $L_{s}^{2}\left( X_{\left( m\right)
}\right) $ (the notation is the same as in \cite{ElPec, PecAop}): $SU_{0}=%
\mathbb{C}$, and, for $l=1,...,m$, $SU_{l}$ is the vector subspace generated
by the functionals of $X_{\left( m\right) }$ of the type%
\begin{equation}
T_{\phi }\left( X_{\left( m\right) }\right) =\sum_{\left\{
k_{1},...,k_{l}\right\} \in V^{\left( m-l,l\right) }}\phi \left( X\left(
k_{1}\right) ,...,X\left( k_{l}\right) \right) \text{,}  \label{USt}
\end{equation}%
for some $\phi \in L\left( V^{\left( n-l,l\right) }\right) $. A random
variable such as (\ref{USt}) is called a $U$\textsl{-statistic} based on $%
X_{\left( m\right) }$, with a \textsl{symmetric kernel} $\phi $ of order $l$.
One has that $SU_{l}\subset SU_{l+1}$ (see e.g. \cite{PecAop}) and $SU_{m}=L_{s}^{2}\left( X_{\left(
m\right) }\right) $. The collection of the \textsl{symmetric Hoeffding spaces%
}\textbf{\ }associated to $X_{\left( m\right) }$, noted $\left\{
SH_{l}:l=0,...,m\right\} $ is defined as follows: $SH_{0}=SU_{0}$, and
\begin{equation*}
SH_{l}=SU_{l}\cap SU_{l-1}^{\perp }\text{,}
\end{equation*}%
where the symbol $\perp $ means orthogonality with respect to the inner
product $\left\langle \cdot ,\cdot \right\rangle _{\mathbf{P}}$ defined in (%
\ref{innPr1}), so that
\begin{equation*}
L_{s}^{2}\left( X_{\left( m\right) }\right) =\bigoplus\limits_{l=0}^{m}SH_{l}%
\text{,}
\end{equation*}%
where the direct sum $\bigoplus $ is again in the sense of $\left\langle
\cdot ,\cdot \right\rangle _{\mathbf{P}}$.

Following \cite[Section 2]{BG1}, we define the real coefficients%
\begin{eqnarray}
d_{l,j} &=&\prod_{r=j}^{l-1}\frac{n-r}{n-r-j}\text{, \ \ }l=2,3,...,m\text{,
}1\leq j\leq l-1\text{,}  \label{coefficz} \\
d_{l,l} &=&N_{l,l}=1\text{, \ \ }l=1,...,m\text{,}  \notag \\
N_{l,j} &=&-\sum_{i=j}^{l-1}\dbinom{l-j}{i-j}d_{l,i}N_{i,j}\text{, \ \ }%
l=2,3,...,m\text{, \ \ }1\leq j\leq l-1.  \notag
\end{eqnarray}

The following result can be proved by using the content of \cite[Section 2]%
{BG1}, or as a special case of \cite[Theorem 11]{PecAop}.

\begin{proposition}
\label{P : supercali}Keep the assumptions and notation of this section.
Then, for $l=1,...,m$, the following assertions are equivalent:

\begin{description}
\item[(i)] $f\left( X_{\left( m\right) }\right) \in SH_{l};$

\item[(ii)] there exists $\phi \in L\left( V^{\left( n-l,l\right) }\right) $
such that%
\begin{equation}
f\left( X_{\left( m\right) }\right) =\sum_{\left\{ k_{1},...,k_{l}\right\}
\in V^{\left( m-l,l\right) }}\phi \left( X\left( k_{1}\right) ,...,X\left(
k_{l}\right) \right) \text{,}  \label{DegUstats}
\end{equation}%
and
\begin{equation*}
\mathbf{E}\left[ \phi \left( X\left( 1\right) ,...,X\left( l\right) \right)
\mid X\left( 1\right) ,...,X\left( l-1\right) \right] =0\text{.}
\end{equation*}
\end{description}

Moreover, for every $h\left( X_{\left( m\right) }\right) \in L_{s}^{2}\left(
X_{\left( m\right) }\right) $, the orthogonal projection of $h\left(
X_{\left( m\right) }\right) $ on $SH_{l}$, $l=1,...,m$, is given by
\begin{equation*}
\mathsf{proj}\left( h\left( X_{\left( m\right) }\right) \mid SH_{l}\right)
=\sum_{\left\{ k_{1},...,k_{l}\right\} \in V^{\left( m-l,l\right) }}\phi
_{h}^{\left( l\right) }\left( X\left( k_{1}\right) ,...,X\left( k_{l}\right)
\right) ,
\end{equation*}%
where, for every $\left\{ j_{1},...,j_{l}\right\} \in V^{\left( n-l,l\right)
}$,%
\begin{eqnarray}
&&\phi _{h}^{\left( l\right) }\left( j_{1},..,j_{l}\right)  \label{Hoeff} \\
&=&d_{m,l}\sum_{a=1}^{l}N_{l,a}\sum_{1\leq i_{1}<...<i_{a}\leq l}\mathbf{E}%
\left[ h\left( X_{\left( m\right) }\right) -\mathbf{E}\left( h\left(
X_{\left( m\right) }\right) \right) \mid X\left( 1\right)
=j_{i_{1}},...,X\left( a\right) =j_{i_{a}}\right] .  \notag
\end{eqnarray}
\end{proposition}

\medskip

The kernel $\phi $ of the $U$-statistic $f\left( X_{\left( m\right) }\right)
$ appearing in (\ref{DegUstats}) is said to be \textsl{completely degenerated%
}. Completely degenerated kernels are related to the notion of \textit{weak
independence} in \cite[Theorem 6]{PecAop}. Note that, in the above quoted
references, the content of Proposition \ref{P : supercali} is proved for
real valued symmetric statistics (the extension of such results to complex
random variables is immediate: just consider separately the real and the
imaginary parts of each statistic). Formula (\ref{Hoeff}) completely
characterizes the symmetric Hoeffding spaces associated to $X_{\left(
m\right) }$: it can be obtained by recursively applying an appropriate
version of the M\"{o}bius inversion formula (see e.g. \cite[Exercise 18,
Section 5.6]{Sagan}), on the lattice of the subsets of $\left[ n\right] $
(see also \cite[Theorem 11]{PecAop}, for a generalization of (\ref{Hoeff})
to the case of Generalized Urn Sequences). In the next section we state and
prove the main result of this note, that is, that the spaces $SH_{l}$, $%
l=1,...,m$, admit a further algebraic characterization in terms of Specht
modules.

\section{Hoeffding spaces and two-blocks Specht modules\label{S : HS//SM}}

\subsection{Main results and some consequences}

The main achievement of this note is the following statement, which is a
more precise reformulation of Theorem \ref{T : Main Intro}, as stated in the
Introduction. The proof is deferred to Section \ref{SS : Proofs !!}.

\begin{theorem}
\label{T : Main}Under the above notation and assumptions, for every $f\left(
X_{\left( m\right) }\right) \in L_{s}^{2}\left( X_{\left( m\right) }\right) $
and every $l=0,1,...,m$, the following assertions are equivalent:

\begin{enumerate}
\item $f\left( X_{\left( m\right) }\right) \in SH_{l};$

\item $f\in K_{l}^{\left( n-m,m\right) }$, where the $\mathfrak{S}_{n}$%
-module $K_{l}^{\left( n-m,m\right) }$ is defined through formula (\ref%
{spechtDec}) (in particular, $K_{l}^{\left( n-m,m\right) }\in \left[
S^{\left( n-l,l\right) }\right] $).
\end{enumerate}
\end{theorem}

We now list some consequences of Theorem \ref{T : Main}. They can be
obtained by properly combining Proposition \ref{P : supercali} with the five
facts (\textbf{a})--(\textbf{e}), as listed at the end of Section \ref{S :
URP}.

\begin{corollary}
Under the above notation and assumptions,

\begin{enumerate}
\item for every $l=1,...,m,$ $f\in M^{\left( n-m,m\right) }$ and $\mathbf{i}%
_{\left( m\right) }=\left\{ i_{1},...,i_{m}\right\} \in V^{\left(
n-m,m\right) }$,%
\begin{align}
& f^{\left( n-l,l\right) }\left( \mathbf{i}_{\left( m\right) }\right)
\label{CEXP} \\
& =\frac{D_{\left( n-l,l\right) }}{n!}\sum_{x\in \mathfrak{S}_{n}}\chi
^{\left( n-l,l\right) }\left( x\right) f\left( x^{-1}\mathbf{i}_{\left(
m\right) }\right) \\
& =\sum_{\left\{ i_{1},...,i_{l}\right\} \subseteq \mathbf{i}_{\left(
m\right) }}d_{m,l}\sum_{a=1}^{l}N_{l,a} \times \notag \\
& \quad\quad\quad\quad\sum_{1\leq s_{1}<...<s_{a}\leq l}%
\mathbf{E}\left[ f\left( X_{\left( m\right) }\right) -\mathbf{E}\left(
f\left( X_{\left( m\right) }\right) \right) \mid X\left( 1\right)
=i_{s_{1}},...,X\left( a\right) =i_{s_{a}}\right] ,  \notag
\end{align}
where $D_{(n-l,l)} = \binom{n}{l} - \binom{n}{l-1}$.
\item for every $l=1,...,m$, every symmetric $U$-statistic, based on $X_{\left(
m\right) }$ and with a completely degenerated kernel of order $l$, has the
form (\ref{CEXP}) for some $f\in M^{\left( n-m,m\right) }$. It follows that $%
SH_{l}$ is an irreducible $\mathfrak{S}_{n}$-module, carrying a
representation in $\left[ S^{\left( n-l,l\right) }\right] $.
\end{enumerate}
\end{corollary}

For instance, by using \cite[Exercice 5.d, p. 87]{Sagan}, we deduce from (%
\ref{CEXP}) that for every $\mathbf{i}_{\left( m\right) }=\left\{
i_{1},...,i_{m}\right\} \in V^{\left( n-m,m\right) }$ and $f\in M^{\left(
n-m,m\right) }$,%
\begin{eqnarray*}
&&\frac{n-1}{n!}\sum_{x\in \mathfrak{S}_{n}}\left\{ \left( \text{number of
fixed points of }x\right) -1\right\} \times f\left( x\mathbf{i}_{\left(
m\right) }\right) \\
&=&\prod_{r=1}^{m-1}\frac{n-r}{n-r-1}\sum_{s=1}^{m}\mathbf{E}\left[ f\left(
X_{\left( m\right) }\right) -\mathbf{E}\left( f\left( X_{\left( m\right)
}\right) \right) \mid X\left( 1\right) =i_{s}\right] .
\end{eqnarray*}

The next result gives an algebraic explanation of a property of degenerated $%
U$-statistics, already pointed out -- in the more general framework of
Generalized Urn Sequences -- in \cite[Corollary 9]{PecAop}. Basically, it
states that the orthogonality, between two completely degenerated $U$%
-statistics of different orders, is preserved after shifting one of the two
arguments. It can be useful when determining the covariance between two $U$%
-statistics based on two urn sequences of different lenghts.

\begin{corollary}
Let $f,h\in M^{\left( n-m,m\right) }$ be such that $f\left( X_{\left(
m\right) }\right) \in SH_{j}$ and $h\left( X_{\left( m\right) }\right) \in
SH_{l}$ for some $1\leq j\neq l\leq m$. Consider moreover an element $%
\mathbf{k}_{\left( m\right) }=\left\{ k_{1},...,k_{m}\right\} \in V^{\left(
n-m,m\right) }$ such that, for some $r=0,...,m$, $\mathsf{Card}\left(
\mathbf{k}_{\left( m\right) }\cap \left\{ 1,...,m\right\} \right) =r$, and
note $X_{\left( m\right) }^{\prime }=\left( X\left( k_{1}\right)
,...,X\left( k_{m}\right) \right) $. Then,%
\begin{equation*}
\mathbf{E}\left( f\left( X_{\left( m\right) }\right) \overline{h\left(
X_{\left( m\right) }^{\prime }\right) }\right) =0.
\end{equation*}
\end{corollary}

\begin{proof}
Due to the exchangeability of the vector $\left( X\left( 1\right)
,...,X\left( n\right) \right) $, we can assume, without loss of generality,
that
\begin{equation*}
\mathbf{k}_{\left( m\right) }=\left\{ 1,...,r,m+1,...,2m-r\right\} .
\end{equation*}%
Now introduce the permutation (written as a product of translations)
\begin{equation}
y=\left( r+1\rightarrow m+1\right) \left( r+2\rightarrow m+2\right) \cdot
\cdot \cdot \left( m\rightarrow 2m-r\right) ,  \label{specperm}
\end{equation}%
and note that
\begin{equation*}
\mathbf{E}\left( f\left( X_{\left( m\right) }\right) \overline{h\left(
X_{\left( m\right) }^{\prime }\right) }\right) =\mathbf{E}\left( f\left(
X_{\left( m\right) }\right) \overline{h\left( (Xy)_{\left( m\right)
}\right) }\right) ,
\end{equation*}%
so that the conclusion derives immediately from formula (\ref{OS1}), by
setting $x=e$ and $y$ as in (\ref{specperm}).
\end{proof}

\subsection{Remaining proofs \label{SS : Proofs !!}}

The key of the proof of Theorem \ref{T : Main} is nested in the following
Lemma.

\begin{lemma}
\label{L : Ugo}Let the previous notation prevail. Then,

\begin{enumerate}
\item for each $l=1,...,m$, a basis of $SU_{l}$ is given by the set of
random variables
\begin{equation*}
\left\{ \eta _{\mathbf{i}_{\left( l\right) }}\left( X_{\left( m\right)
}\right) :\mathbf{i}_{\left( l\right) }\in V^{\left( n-l,l\right) }\right\}
\text{,}
\end{equation*}%
where, for each $\mathbf{k}_{\left( m\right) }\in V^{\left( n-m,m\right) }$,%
\begin{equation}
\eta _{\mathbf{i}_{\left( l\right) }}\left( \mathbf{k}_{\left( m\right)
}\right) =\left\{
\begin{array}{lll}
1 &  & \text{if }\mathbf{i}_{\left( l\right) }\subseteq \mathbf{k}_{\left(
m\right) } \\
0 &  & \text{otherwise;}%
\end{array}%
\right.  \label{basisU}
\end{equation}

\item for each $l=1,...,m$, the restriction of the action (\ref{actionSym})
of $\mathfrak{S}_{n}$ to the vector subspace of $M^{\left( n-m,m\right) }$
generated by the set $\{\eta _{\mathbf{i}_{\left( l\right) }}:\mathbf{i}%
_{\left( l\right) }\in V^{\left( n-l,l\right) }\}$, defined in (\ref{basisU}%
), is equivalent to the action carried by the $\mathfrak{S}_{n}$-module $%
M^{\left( n-l,l\right) }$.
\end{enumerate}
\end{lemma}

\begin{proof}
Fix $l=1,...,m$, and observe that, for every $\mathbf{i}_{\left( l\right)
}\in V^{\left( n-l,l\right) }$,%
\begin{equation*}
\eta _{\mathbf{i}_{\left( l\right) }}\left( X_{\left( m\right) }\right)
=\sum_{\left\{ k_{1},...,k_{l}\right\} \in V^{\left( m-l,l\right) }}\mathbf{1%
}_{\mathbf{i}_{\left( l\right) }}\left( \left\{ X\left( k_{1}\right)
,...,X\left( k_{l}\right) \right\} \right) \text{,}
\end{equation*}%
so that the first part of the statement follows from the definition of $%
SU_{l}$, and the fact that every $\phi \in V^{\left( m-l,l\right) }$ is a
linear combination of functions of the type $\mathbf{1}_{\mathbf{i}_{\left(
l\right) }}\left( \cdot \right) $. To prove the second part, first recall
that a basis of the $\mathfrak{S}_{n}$-module $M^{\left( n-l,l\right) }$ is
given by the set $\left\{ \mathbf{1}_{\mathbf{i}_{\left( l\right) }}\left(
\cdot \right) :\mathbf{i}_{\left( l\right) }\in V^{\left( n-l,l\right)
}\right\} $, and that the action of $\mathfrak{S}_{n}$ on $M^{\left(
n-l,l\right) }$ is completely described by the action%
\begin{equation*}
x\mathbf{1}_{\mathbf{i}_{\left( l\right) }}=\mathbf{1}_{x\mathbf{i}_{\left(
l\right) }}.
\end{equation*}%
We can therefore construct a $\mathfrak{S}_{n}$-isomorphism between $\left\{ \eta
_{\mathbf{i}_{\left( l\right) }}:\mathbf{i}_{\left( l\right) }\in V^{\left(
n-l,l\right) }\right\} $ and $M^{\left( n-l,l\right) }$ by linearly
extending the mapping
\begin{equation*}
\tau \left( \eta _{\mathbf{i}_{\left( l\right) }}\right) =\mathbf{1}_{%
\mathbf{i}_{\left( l\right) }},\text{ \ \ }\mathbf{i}_{\left( l\right) }\in
V^{\left( n-l,l\right) }\text{,}
\end{equation*}%
and by observing that, for every $\mathbf{k%
}_{\left( m\right) }\in V^{\left( n-m,m\right) }$, $\mathbf{i}_{\left(
l\right) }\in V^{\left( n-l,l\right) }$ and $x\in \mathfrak{S}_{n}$,%
\begin{equation*}
x\eta _{\mathbf{i}_{\left( l\right) }}\left( \mathbf{k}_{\left( m\right)
}\right) =\eta _{\mathbf{i}_{\left( l\right) }}\left( x^{-1}\mathbf{k}%
_{\left( m\right) }\right) =\eta _{x\mathbf{i}_{\left( l\right) }}\left(
\mathbf{k}_{\left( m\right) }\right) \text{.}
\end{equation*}
This concludes the proof.
\end{proof}

\bigskip

\noindent \textbf{End of the proof of Theorem \ref{T : Main}. }Since $%
SU_{0}=SH_{0}=K_{0}^{\left( n-m,m\right) }=\mathbb{C}$, the relation between
representations%
\begin{equation*}
M^{\left( n-l,l\right) }\cong S^{\left( n\right) }\oplus S^{\left(
n-1,1\right) }\oplus \cdot \cdot \cdot \oplus S^{\left( n-l,l\right) },\text{
\ \ }\forall l=1,...,m\text{,}
\end{equation*}%
along with Lemma \ref{L : Ugo}, implies that the restriction of the action (%
\ref{actionSym}) of $\mathfrak{S}_{n}$ to those $f\in L\left( V^{\left(
n-m,m\right) }\right) $ such that $f\left( X_{\left( m\right) }\right) \in
SH_{l}$ is an element of $\left[ S^{\left( n-l,l\right) }\right] $. This
yields that each one of the $m+1$ summands in the decomposition
\begin{equation*}
M^{\left( n-m,m\right) }=\mathbb{C}\oplus \bigoplus_{l=1}^{m}\left\{
f:f\left( X_{\left( m\right) }\right) \in SH_{l}\right\}
\end{equation*}%
is an irreducible $\mathfrak{S}_{n}$-submodule of $M^{\left( n-m,m\right) }$%
. Since the decomposition (\ref{spechtDec}) of $M^{\left( n-m,m\right) }$ is
unique, this gives%
\begin{equation*}
\left\{ f:f\left( X_{\left( m\right) }\right) \in SH_{l}\right\}
=K_{l}^{\left( n-m,m\right) }\text{,}
\end{equation*}%
as required. $\blacksquare $

\medskip

\noindent {\bf Acknowledgement.} We are grateful to Omar El-Dakkak for useful comments.

\end{document}